\documentclass{birkjour}
\usepackage{graphicx} 
\usepackage{amsmath,amssymb,amsthm}
 \newtheorem{thm}{Theorem}[section]
 \newtheorem{cor}[thm]{Corollary}

 \theoremstyle{definition}
 \newtheorem{defn}[thm]{Definition}
 \theoremstyle{remark}

 \numberwithin{equation}{section}

\begin{document}

%
%
%
%
%
%
%
%
%

\title[Fibonacci-Lucas Spinors Obtained from Hybrid Numbers]
 {Fibonacci-Lucas Spinors Obtained from Hybrid Numbers}

\author{Selime Beyza Özçevik}
\address{Department of Mathematics,\br
Ondokuz Mayıs University,\br
55139, Samsun,\br
Turkey}
\email{ozcevikbeyza8@gmail.com}

\author{Abdullah Dertli}
\address{Department of Mathematics,\br
Ondokuz Mayıs University,\br
55139, Samsun,\br
Turkey}
\email{abdullah.dertli@gmail.com}

\author{Savaş Karaahmetoğlu}

\address{%
Akpınar Science High School,\\
55760, Samsun\\
Turkey}

\email{sawasx@gmail.com}






\date{January 1, 2004}

\begin{abstract}
The aim of this work is to provide the contributors to journals or Hybrid numbers, akin to spinors, possess a broad range of applications in mathematical physics, geometry, and mathematics. In this study, these two significant topics were collectively addressed, introducing a new perspective to spinors and defining hybrid spinors, from which several basic properties were derived. Furthermore, by considering the Fibonacci and Lucas number sequences, subjects widely explored in number theory, the study defined hybrid Fibonacci and Lucas spinor sequences, examining some of their properties. Lastly, hybrid Fibonacci and Lucas polynomial spinor sequences have been introduced through the utilization of the polynomials associated with these number sequences.
\end{abstract}

\maketitle
\section{Introduction}
Complex, hyperbolic, and dual numbers represent established two-dimensional number systems that have garnered considerable attention from researchers, particularly in the past century. Numerous scholars have explored the geometric and physical applications inherent in these numerical constructs. Much like how complex numbers facilitate the description of Euclidean plane geometry, hyperbolic numbers find application in characterizing the geometry of the Minkowski plane, while dual numbers are employed to articulate the geometry of the Galilean plane \cite{yaglom2012simple}. n scholarly discourse, hyperbolic numbers are denoted by various terms such as 'spacetime,' 'double,' 'perplex,' and 'split-complex numbers,' as indicated in the literature \cite{borota2000spacetime,fjelstad1986extending, poodiack2009fundamental, rochon2004algebraic}. This number system constitutes a generalization encompassing the characteristics of complex, hyperbolic, and dual number systems. These mathematical structures, utilized in diverse disciplines, exhibit various relationships among themselves, underscoring the significance of understanding and applying mathematical diversity.Hybrid numbers are a mathematical structure that typically combines real numbers with complex numbers, exhibiting broad applications in mathematics. This structure provides an extended version of the complex plane. Simultaneously, spinors are mathematical objects playing a crucial role in quantum mechanics and theoretical physics. Represented often by complex numbers or matrices, spinors are utilized to model the rotational motion of particles around their own axis. Hybrid numbers and spinors are structural elements of critical importance in mathematical physics and theoretical physics, serving to comprehend the fundamental properties of the universe and explain the behaviors of fundamental particles.

\section{Preliminaries}
Hybrid numbers were initially introduced in \cite{ozdemir2018introduction}. Drawing upon this study, we have incorporated essential concepts necessary for our own research, as outlined below. Further information on hybrid numbers can be found in \cite{ozdemir2018introduction}.

The set of hybrid numbers, denoted by $\mathbb{K}$, is defined as

$$
\mathbb{K}=\left\{a+b i+c \varepsilon+d h ; a, b, c, d \in \mathbb{R}, i^2=-1, \varepsilon^2=0, h^2=1, i h=-h i=\varepsilon+i^2\right\}.
$$

This set of numbers can be thought as a set of quadruplet. And a real, complex, dual and hyperbolic units are defined as

$$
1 \longleftrightarrow(1,0,0,0), i \leftrightarrow(0,1,0,0), \varepsilon \leftrightarrow(0,0,1,0), h \longleftrightarrow(0,0,0,1)
$$
respectively.

Let $Z,W \in\mathbb{K}$ be two hybrid numbers. The sum of two hybrid numbers is defined by summing their components. The Hybridian product

$$
 ZW=\left(a_1+b_1 i+c_1 \varepsilon+d_1 h\right)\left(a_2+b_2 i+c_2 \varepsilon+d_2 h\right)
$$
is obtained by distributing the terms on the right.

The conjugate of a hybrid number $Z=a+b i+c \varepsilon+d h$, denoted by $\bar{Z}$, is defined as
$$
\bar{Z}=a-b i-c \varepsilon-d h.
$$
The scalar product in the hybrid numbers is defined as follows

$$
\begin{aligned}
& g(p, q): K \times K \longrightarrow \mathbb{R} & \\
& g\left(Z_1, Z_2\right) =\frac{Z_1 \overline{Z_2} +Z_2 \overline{Z_1}}{2} \\
& =a_1 a_2+b_1 b_2-b_1 c_2-b_2 c_1-d_1 d_2
\end{aligned}
$$
where $Z_1=a_1+b_1 i+c_1 \varepsilon+d_1 h$ and $Z_2=a_2+b_2 i+c_2 \varepsilon+d_2 h$.
On the other hand, the vector product of two hybrid numbers $Z_1$ and $Z_2$ defined as

$$
\begin{aligned}
& x: K \times K \longrightarrow K \\
& Z_1 \times Z_2=\frac{Z_1 \overline{Z_2}-Z_2 \overline{Z_1}}{2}.
\end{aligned}
$$

\section{Hybrid Spinors}
In this section, we will introduce a new relationship between hybrid numbers and spinors. 
Let the set of Hybrid numbers be $\mathbb{K}$. Let us define the transformation that gives the relation between spinors and hybrid numbers as follows

$$
\chi :\mathbb{K} \to \mathbb{HS}
$$
$$
\chi \left( {a + bi + c\varepsilon  + dh} \right) = \left[ {\begin{array}{*{20}{c}}
{a + hc}\\
{\left( {c - b} \right) + hd}
\end{array}} \right].
$$

If we take that conjugate of the hybrid number $z$ is $z^*=a-b i-c \varepsilon-d h$, the hybrid spinor corresponding to conjugate hybrid number can be written as

$$
\chi \left( {a - bi - c\varepsilon  - dh} \right) = \left[ {\begin{array}{*{20}{c}}
{a - hc}\\
{\left( {b - c} \right) - hd}
\end{array}} \right] \equiv \overline {SH}.
$$

We associated to the product of two hybrid numbers $p \times q$, which is equal to the hybrid number -matrix product $L(p) q$, with a spinor matrix product as follows

$$
p \times q  \rightarrow L(p) q \rightarrow-h \hat{Q} H 
$$
where
$$
\hat{Q}  =\left[\begin{array}{cc}
a+h c & (b-c)+h d \\
(c-b)+h d & a-h c
\end{array}\right],
$$
and 
$$
H  =\left[\begin{array}{cc}
1 & 0 \\
0 & -1
\end{array}\right].
$$

\section{Hybrid Fibonacci and Lucas Spinors}
In this section, we will present the hybrid Fibonacci and Lucas spinors using a similar methodology. Additionally, we will investigate the characteristics of hybrid Fibonacci and Lucas spinors, deriving theorems and formulas that define these spinors.

The Fibonacci sequence \cite{koshy2019fibonacci} is defined by the recurrence relation

 $$F_{n+1}=F_n+F_{n-1}$$
with initial conditions $F_0=0, F_1=1$.

The Lucas sequence \cite{koshy2019fibonacci} is defined by the recurrence relation

 $$L_{n+1}=L_n+L_{n-1}$$
with initial conditions $L_0=2, L_1=1$.

$z=F_n+F_{n+1} i+F_{n+2} \varepsilon+F_{n+3} h$ be a  hybrid Fibonacci  number Let us define the transformation that gives the relation between  hybrid Fibonacci numbers and hybrid spinous as follows

$$
\chi: \mathbb{F} \longrightarrow \mathbb{HS} 
$$
$$
\begin{aligned}
\chi \left( {{F_n} + {F_{n + 1}}i + {F_{n + 2}}\varepsilon  + {F_{n + 3}}h} \right) = \left[ {\begin{array}{*{20}{c}}
{{F_n} + h{F_{n + 2}}}\\
{\left( {{F_{n + 2}} - {F_{n + 1}}} \right) + h{F_{n + 3}}}
\end{array}} \right] \equiv \overline {FSH_{n}}.
\end{aligned}
$$

We say that a new sequence is hybrid Fibonacci spinor sequence. Note that for $n \geq 0$,

$$
FSH_{n+2}=FSH_{n+1}+FSH_n
$$
can be written.
\begin{defn}
Let the conjugate of the hybrid Fibonacci number $HF_{n}$ be $\overline {H{F_n}}  = {F_n} - i{F_{n + 1}} - \varepsilon {F_{n + 2}} - h{F_{n + 3}}$. Furthermore from the correspondence, the  hybrid Fibonacci spinor $\overline {FSH_n} $ corresponding to the conjugate of the hybrid Fibonacci number is written by

$$
\chi \left( {\overline {H{F_n}} } \right) = \left[ {\begin{array}{*{20}{c}}
{{F_n} - h{F_{n + 2}}}\\
{\left( {{F_{n + 1}} - {F_{n + 2}}} \right) - h{F_{n + 3}}}
\end{array}} \right] = \overline {FS{H_n}}.
$$

Additionally, utilizing the conjugate definitions provided in the second section, we can provide the following conjugate definitions.
The hyperbolic conjugate of  hybrid Fibonacci spinor $FS{H_n}$ is

$$
FSH_n^* = \left[ {\begin{array}{*{20}{c}}
{{F_n} - h{F_{n + 2}}}\\
{\left( {{F_{n + 2}} - {F_{n + 1}}} \right) - h{F_{n + 3}}}
\end{array}} \right].
$$

 Hybrid Fibonacci spinor conjugate $F\tilde S{H_n} = hH\overline {FS{H_n}}$ of  hybrid Fibonacci spinor $FS{H_n}$ is

$$
F\tilde S{H_n} = \left[ {\begin{array}{*{20}{c}}
{ - {F_{n + 2}} + h{F_n}}\\
{{F_{n + 3}} + h\left( {{F_{n + 1}} - {F_{n + 2}}} \right)}
\end{array}} \right].
$$

The mate of  hybrid Fibonacci spinor $F\mathop S\limits^ \vee  {H_n} =  - H\overline {FS{H_n}}$ is

$$
F\mathop S\limits^ \vee  {H_n} = \left[ {\begin{array}{*{20}{c}}
{ - {F_n} + h{F_{n + 2}}}\\
{\left( {{F_{n + 2}} - {F_{n + 1}}} \right) - h{F_{n + 3}}}
\end{array}} \right]
$$
where $H=\left[\begin{array}{cc}1 & 0 \\ 0 & -1\end{array}\right]$.
\end{defn}

\begin{cor}
For the  hybrid Fibonacci spinor $FSH_n$ and its conjugates, following formulas are valid

\begin{itemize}
    \item $FS{H_n} + \overline {FS{H_n}}  = 2\left[ {\begin{array}{*{20}{c}}
{{F_n}}\\
0
\end{array}} \right]$,
\item $FS{H_n} + FS{H_n}^* = 2\left[ {\begin{array}{*{20}{c}}
{{F_n}}\\
{{F_n}}
\end{array}} \right] = 2{F_n}\left[ {\begin{array}{*{20}{c}}
1\\
1
\end{array}} \right]$,
\item $FS{H_n} + F\mathop S\limits^ \vee  {H_n} = 2\left[ {\begin{array}{*{20}{c}}
{h{F_{n + 2}}}\\
{{F_n}}
\end{array}} \right]$,
\item $hF\widetilde S{H_n} =  - F\mathop S\limits^ \vee  {H_n}$.
\end{itemize}
\end{cor}

\begin{thm}
The Binet formula for hybrid  Fibonacci spinor sequence is as follows:

$$
FSH_n=\left[\begin{array}{l}
\frac{1}{\sqrt{5}}+h\left(\frac{3+\sqrt{5}}{2 \sqrt{5}}\right)  \\
\frac{1}{\sqrt{5}}+h\left(\frac{2+\sqrt{5}}{\sqrt{5}}\right) 
\end{array}\right] \alpha^n+\left[\begin{array}{l}
-\frac{1}{\sqrt{5}}-h\left(\frac{3-\sqrt{5}}{2 \sqrt{5}}\right)  \\
-\frac{1}{\sqrt{5}}-h\left(\frac{2-\sqrt{5}}{\sqrt{5}}\right) 
\end{array}\right] \beta^n
$$
where $\alpha  = \frac{{1 + \sqrt 5 }}{2}$ and  $\beta  = \frac{{1 - \sqrt 5 }}{2}$.
\end{thm}

\begin{proof}
The characteristic equation of the recurrence relation of hybrid Fibonacci spinor sequence is $x^2-x-1=0$. The roots of are this equation are $\alpha=\frac{1+\sqrt{5}}{2}, \beta=\frac{1-\sqrt{5}}{2}$.
The Binet's formula for hybrid Fibonacci spinor is as follows
$$
FSH_n=S H_A \alpha^n+S H_B \beta^n
$$
where $S H_A=\left[\begin{array}{l}a_1 \\ a_2\end{array}\right], \mathrm{SH}_B=\left[\begin{array}{l}b_1 \\ b_2\end{array}\right] \quad 2 \times 1$ matrices.
If we write this equation for $n=0$ and $n=1$, we obtain the Binet formula as

$$
FSH_n=\frac{1}{2 \sqrt{5}}\left(\left[\begin{array}{l}
2+(3 \sqrt{5}) h \\
2+(4+2 \sqrt{5}) h
\end{array}\right] \alpha^n-\left[\begin{array}{c}
2+(3-\sqrt{5}) h \\
2+(4-2 \sqrt{5}) h
\end{array}\right] \beta^n\right) \text {. }
$$
\end{proof}

\begin{thm}
The generating function for hybrid Fibonacci spinor $FSH_n$ is

$$
{G_{FS{H_n}}}\left( x \right) = \frac{1}{{1 - x - {x^2}}}\left[ {\begin{array}{*{20}{c}}
{1 + h\left( {3 - x} \right)}\\
{1 + h\left( {5 - 2x} \right)}
\end{array}} \right].
$$
\end{thm}

\begin{proof}
The generating function $h(x)$ for hybrid Fibonacci spinor can be writer as follows:
$$
h(x)=\sum_{n=0}^{\infty} FSH_n x^n
$$

If the recurrence relation of hybrid Fibonacci spinor sequence is considered, we have
$$
\sum_{n=0}^{\infty} FSH_{n+2} x^n=\sum_{n=0}^{\infty} FSH_{n+1} x^n+\sum_{n=0}^{\infty} FSH_n x^n .
$$

Then, we can write
$$
\sum_{n=2}^{\infty} FSH_n x^{n-2}=\sum_{n=1}^{\infty} FSH_n x^{n-1}+\sum_{n=0}^{\infty} FSH_n x^n
$$
$$
\frac{1}{x^2}\left(-FSH_0-FSH_1+h(x)\right)=\frac{1}{x}\left(-FSH_0+h(x)\right)+h(x) .
$$

Since the first two terms of hybrid Fibonacci spinor sequence are
$$
\mathrm{FSH}_0=\left[\begin{array}{c}
h \\
2 h
\end{array}\right], \mathrm{FSH}_1=\left[\begin{array}{l}
1+2 h \\
1+3 h
\end{array}\right]
$$
we get,
$$
h(x)=\frac{1}{1-x-x^2}\left[\begin{array}{l}
1+h(3-x) \\
1+h(5-2 x)
\end{array}\right] .
$$
\end{proof}

The hybrid Fibonacci spinor matrix is defined as follows

$$
Q_h=\left[\begin{array}{lr}
F_n+h F_{n+2} & \left(F_{n+1}-F_{n+2}\right)+h F_{n+3} \\
\left(F_{n+2}-F_{n+1}\right)+h F_{n+3} & F_n-h F_{n+2}
\end{array}\right].
$$

\begin{thm}
For $n \geq 1$, we have Cassini formula for hybrid Fibonacci spinous
$$
\operatorname{det}\left(Q_n\right)=-F_{2 n+5}+2 F_n^2.
$$
\end{thm}

Now, we express hybrid Lucas spinous.
Let $z_n^{\prime}=L_n+L_{n+1} i+L_{n+2} \varepsilon+L_{n+3} h$ be the $n t h$ Lucas hybrid number where $L_n$ nth Lucas number.

Hybrid Lucas spinor$ LSH_n$ corresponding to $n t h$  hybrid Lucas number $z_n^{\prime}$ is
$$
\begin{aligned}
z_n^{\prime} \rightarrow f\left(L_n+L_{n+1} i+L_{n+2} \varepsilon+L_{n+3} h\right) & =\left[\begin{array}{l}
L_n+h L_{n+2} \\
L_n+h L_{n+3}
\end{array}\right] \\
& \equiv LSH_n .
\end{aligned}
$$
\begin{thm}
 For $n \geq 0$, the Binet formula for $n$th  hybrid Lucas spinor $LSH_n$ is \\
$$
LSH_n=A \alpha^n+B \beta^n,
$$
where
$$
A=\left[\begin{array}{c}
\frac{2+h(3 \sqrt{5})}{2} \\
1+(2+\sqrt{5}) h
\end{array}\right], \quad B=\left[\begin{array}{l}
\frac{3-\sqrt{5} h}{2} \\
1+(2-\sqrt{5}) h
\end{array}\right].
$$
\end{thm}

\begin{thm}
The generating function for hybrid Lucas spinor $LSH_n$ is \\
$$
g^{\prime}(x)=\frac{1}{1-x-x^2}\left(\left[\begin{array}{l}
3+7 h \\
3+11 h
\end{array}\right]-x\left[\begin{array}{l}
2+3 h \\
2+4 h
\end{array}\right]\right)
$$.
\end{thm}

\begin{thm}
For $n \geq 2$, we have
\begin{description}
\item[i]  $FSH_{n+2}-FSH_{n-2}=LSH_n$, \\
\item[ii] $5 FSH_n+LSH_n=2LSH_{n+1}$.
\end{description}
\end{thm}

\begin{proof}
$$
\begin{aligned} (i) 
\operatorname{FSH} & =\frac{1}{\sqrt{5}}\left(A \alpha^n-B \beta^n\right) \\
FSH_{n+2}-FSH_{n-2} & =\frac{1}{\sqrt{5}}\left(A \alpha^n\left(\alpha^2+\alpha^{-2}\right)-B \beta^n\left(\beta^2-\beta^{-2}\right)\right) \\
& =\frac{1}{\sqrt{5}}\left(A \alpha^n\left(\frac{3+\sqrt{5}}{2}-\frac{3-\sqrt{5}}{2}\right)-B \beta^n\left(\frac{3-\sqrt{5}}{2}-\frac{3+\sqrt{5}}{2}\right)\right) \\
& =\frac{1}{\sqrt{5}}\left(A \alpha^n-\sqrt{5}+B \beta^n \sqrt{5}\right)=A \alpha^n+B \beta^n=LSH_n .
\end{aligned}
$$

The second identity is proved in a similar way as the frst.
\end{proof}

\begin{thm}
For $n \geq 1$, we have
\begin{description}
\item[i]  $L_{n+1} FSH_n+L_n FSH _{n-1}=LSH_{2 n}$,\\
\item[ii] $F_{n+1} FSH_n+F_n FSH_{n-1}=FSH_{2 n}$.
\end{description}
\end{thm}

\begin{proof}
(i) The Binet formula for Lucas sequence is $L_{n}=\alpha ^{n}+\beta ^{n}$. We
can write
\begin{eqnarray*}
LSH_{2n} &=&\frac{A\left( \alpha ^{2n+1}-\beta \alpha ^{2n}\right) +B(\alpha
\beta ^{2n}-\beta ^{2n+1})}{\alpha -\beta } \\
&=&\frac{\left( A\alpha ^{n}-B\beta ^{n}\right) (\alpha ^{n}\alpha +\beta
^{n}\beta )}{2\left( \alpha -\beta \right) }+\frac{\left( A\alpha
^{n-1}-B\beta ^{n-1}\right) (\alpha ^{n}+\beta ^{n})}{2\left( \alpha -\beta
\right) } \\
&=&(\alpha ^{n+1}+\beta ^{n+1})\left( \frac{1}{2(\alpha -\beta )}\left(
A\alpha ^{n}-B\beta ^{n}\right) \right)  \\
&&+\left( \alpha ^{n}+\beta ^{n}\right) \left( \frac{1}{2(\alpha -\beta )}%
\left( A\alpha ^{n-1}-B\beta ^{n-1}\right) \right)  \\
&=&L_{n+1}FSH_{n}+L_{n}FSH_{n-1}.
\end{eqnarray*}

(ii) The Binet formula for Fibonacci sequence is
$$
F_n=\frac{\alpha^n-\beta^n}{\alpha-\beta},
$$
where $\alpha=\frac{1+\sqrt{5}}{2}, \beta=\frac{1-\sqrt{5}}{2}$. Then, $w e$ can write
$$
\begin{aligned}
& F_{n+1} FSH_n+F_n FSH_{n-1}= \\
& \frac{\alpha^{n+1}-\beta^{n+1}}{\alpha-\beta} \cdot \frac{1}{\sqrt{5}}\left(A \alpha^n-B \beta^n\right)+\frac{\alpha^n-\beta^n}{\alpha-\beta} \cdot \frac{1}{\sqrt{5}} \cdot\left(A \alpha^{n-1}-B \beta^{n-1}\right) \cdot \\
& =\frac{\left(\alpha^{n+1}-\beta^{n+1}\right)\left(A \alpha^n-B \beta^n\right)}{(\alpha-\beta)^2}+\frac{\left(\alpha^n-\beta^n\right)\left(A \alpha^{n-1}-B \beta^{n-1}\right)}{(\alpha-\beta)^2} \\
& =\frac{(\alpha-\beta)\left(A \alpha^{2 n}-B \beta^{2 n}\right)}{(\alpha-\beta)^2} \\
& =FSH_{2 n} .
\end{aligned}
$$
\end{proof}

\section{The hybrid Fibonacci  polynomial spinor}
 Hoggatt, Philips and Leonard defined the Fibonacci and Lucas Polynomials and they have obtained some more identities involving these polynomials, \cite{hoggatt1971twenty}.

Fibonacci polynomial sequence is defined by the recurrence relation
$$
F_{n+1}(x)=x F_n(x)+F_{n-1}(x)
$$
with initial conditions $F_0(x)=0, F_1(x)=1$.
The first few terms of Fibonacci polynomial sequence are as follows
$$
0,1, x, x^2+1, x^3+2 x, x^4+3 x^2+1, \ldots.
$$
The characteristic equation of this sequence
$$
F_n(t)=t^2-x t-1
$$

Lucas polynomial sequence is defined by the recurrence relation
$$
L_{n+1}(x)=x L_n(x)+L_{n-1}(x)
$$
with initial conditions $L_0(x)=2, L_1(x)=1$.
The first few terms of Lucas polynomial sequence are as follows.
$$
2, x, x^2+2, x^3+3 x, x^4+4 x^2+2, \ldots.
$$

We define hybrid Fibonacci  polynomial and transformation of hybrid Fibonacci polynomial spinor.
$$
Z(x)=F_n(x)+F_{n+1}(x) i+F_{n+2}(x) \varepsilon+F_{n+3}(x) h,
$$

$$
\begin{array}{r}
FSH_n(x)=\left[\begin{array}{l}
F_n(x)+h F_{n+2}(x) \\
\left(F_{n+2}(x)-F_{n+1}(x)\right)+h F_{n+3}(x)
\end{array}\right].
\end{array}
$$

The Fibonacci hybrid polynomial spinor sequence is defined by the recurrence relation
$$
{FSH}_{n+2}(x)={FSH}_{n+1}(x)+{FSH}_n(x) .
$$
The following theorems can be proved in a similar way as presented in Section 4.
\begin{thm}
The Binet formula for hybrid Fibonacci  polynomial sequence

$$
 \quad \text { FSH n }(x)=\frac{1}{2}\left(A(x) \alpha(x)^n-B(x) \beta(x)^n\right), \\
$$
where
$$
A(x)=\left[\begin{array}{l}
\frac{2+2 h+h x^2}{\sqrt{x^2+4}}+h x \\
\frac{1+x h}{\sqrt{x^2+4}}+h x^2+h
\end{array}\right], B(x)=\left[\begin{array}{l}
\frac{2+2 h+2 h x^2}{\sqrt{x^2+4}}-h x \\
\frac{1+x h}{\sqrt{x^2+4}}-h x^2-h
\end{array}\right].
$$
\end{thm}

\begin{thm}
The generating function for hybrid Fibonacci polynomial spinor
$$
g(t)=\frac{1}{1-x t-t^2}\left[\begin{array}{l}
1+h\left(x^2+x^2 t+x+1\right) \\
1+h\left(x^3+x^3 t+x^2+2 x+x t+1\right)
\end{array}\right]
$$
\end{thm}

Now, we define hybrid Lucas polynomial and transformation of hybrid Lucas polynomial spinor.
$$
\begin{array}{r}
z^{\prime}(x)=L_n(x)+L_{n+1}(x) i+L_{n+2}(x) \varepsilon+L_{n+3}(x) h, \\

LSH_n(x)=\left[\begin{array}{l}
LSH_n(x)+h LSH_{n+2}(x) \\
LSH_n(x)+h LSH_{n+3}(x)
\end{array}\right]. \\
\end{array}
$$
The hybrid Lucas polynomial spinor sequence is defined by the recurrence relation
$$
{LSH}_{n+2}(x)={LSH}_{n+1}(x)+LSH_n(x).
$$

\begin{thm}
The Binet formula for hybrid Lucas polynomial spinor
$$
\begin{aligned}
& LSH_n(x)=\frac{1}{2 c}\left(\left[\begin{array}{l}
p(x)+2 c+h\left(c x^2+2 c\right) \\
q(x)+2 c+h\left(c x^3+3 x c\right)
\end{array}\right] \alpha^n+\right. \\
& {\left[\begin{array}{l}
-p(x)+2 c+h\left(c x^2+2 c\right) \\
-q(x)+2 c+h\left(c x^3+3 x c\right)
\end{array}\right] \beta^n, }
\end{aligned}
$$
where $c=\sqrt{x^2+4}, p(x)=h x^3-2 h x+b x$, $q(x)=h x^4+5 x^2 h+4 h$.
\end{thm}

\begin{thm}
The generating function for hybrid Lucas polynomial spinor as follows:
$$
g^{\prime}(t)=\frac{1}{1-x t-t^2}\left(\left[\begin{array}{l}
2+x+h\left(x^3+x^2+3 x+2\right. \\
2+x+h\left(x^4+x^3+4 x^2+3 x+2\right.
\end{array}\right]-x t\left[\begin{array}{l}
2+h\left(x^2+2\right) \\
2+h\left(x^3+3 x\right)
\end{array}\right] .\right.
$$
\end{thm}

\begin{thm}
For $n \geq 1$, we have
\begin{itemize}
\item $LSH_n(x)=FSH_{n+1}(x)+FSH_{n-1}(x)$, \\
\item $LSH_n(x)=2 FSH_{n+1}(x)-x \cdot FSH_n(x)$,\\
\item $(a-b)^2 FSH_n(x)=LSH_{n+1}(x)+(-1)^{n+1} LSH_{n-1}(x)$.
\end{itemize}
\end{thm}

\section{Conclusion}
In this study, hybrid spinors were defined using the concept of hybrid numbers. Subsequently, hybrid Fibonacci and hybrid Lucas spinor sequences were defined by employing terms and polynomials of Fibonacci and Lucas number sequences. Fundamental properties such as generator functions, Binet formula, and others were derived for these hybrid sequences. Finally, utilizing the polynomials associated with these sequences, hybrid Fibonacci and hybrid Lucas polynomial spinors were introduced, and similar equations were provided. As a result, this study presents a new perspective for future research by amalgamating concepts crucial in algebra, number theory, and related fields, along with mathematical entities like hybrid numbers and spinors, commonly utilized in areas such as geometry and mathematical physics.

\section*{Declaration of competing interest}
The authors declare that they have no known competing financial interests or personal relationships that could have appeared to influence the work reported in this paper.

\end{document}